\documentclass[
final
]{dmtcs-episciences}

\usepackage[utf8]{inputenc}
\usepackage{subfigure}
\usepackage[T1]{fontenc}

\usepackage{amsmath}
\usepackage{amsthm}

\usepackage{pstricks}
\usepackage{pst-plot, pst-node, pst-text, pst-tree}
\usepackage{tikz,pgfplots}

\usepackage{cancel}
\usepackage{rotating} 

\usepackage{algorithm}
\usepackage{algpseudocode}

\usepackage[round]{natbib}

\newcommand{\I}{{\mathbf I}}
\newcommand{\T}{{\mathbf T}}
\newcommand{\A}{{\mathbf A}}

\newtheorem{theorem}{Theorem}[section]
\newtheorem{proposition}[theorem]{Proposition}
\newtheorem{lemma}[theorem]{Lemma}
\newtheorem{definition}[theorem]{Definition}
\newtheorem{corollary}[theorem]{Corollary}

\newtheorem{open}[theorem]{Open Question}
\newtheorem{example}[theorem]{Example}

\author{Toufik Mansour\affiliationmark{1}
  \and Howard Skogman\affiliationmark{2}
  \and Rebecca Smith\affiliationmark{3}}
\title{Sorting inversion sequences}
\affiliation{
  Department of Mathematics, University of Haifa, Israel\\
  Excellus BlueCross BlueShield, Rochester, NY, USA\\
  Department of Mathematics, SUNY Brockport, NY, USA}
\begin{document}
\publicationdata{vol. 27:1, Permutation Patterns 2024}
{2025}{1}
{10.46298/dmtcs.14010}
{2024-07-31; 2024-07-31; 2025-01-07; 2025-01-14}
{2025-01-29}
\maketitle

\newcommand{\Av}{\operatorname{Av}}

\def\sdwys #1{\xHyphenate#1$\wholeString}
\def\xHyphenate#1#2\wholeString {\if#1$%
\else\say{\ensuremath{#1}}\hspace{2pt}%
\takeTheRest#2\ofTheString
\fi}
\def\takeTheRest#1\ofTheString\fi
{\fi \xHyphenate#1\wholeString}
\def\say#1{\begin{turn}{-90}\ensuremath{#1}\end{turn}}

\newenvironment{onestack}
{
	\begin{scriptsize}
	\psset{xunit=0.0355in, yunit=0.0355in, linewidth=0.02in}
	\begin{pspicture}(0,-2)(32,20)
	\psline{c-c}(0,15)(10,15)
	\psline{c-c}(13,15)(13,2)(19,2)(19,15)
	\psline{c-c}(22,15)(32,15)
	\rput[l](-0.5,12.5){\mbox{output}}
	\rput[r](32,12.5){\mbox{input}}
}
{
	\end{pspicture}
	\end{scriptsize}
}
\newcommand{\fillstack}[3]{%
	\rput[l](-0.5,17.5){\ensuremath{#1}}
	\rput[c](16.1, 7.5){\begin{sideways}{\sdwys{#2}}\end{sideways}}
	\rput[r](32,17.5){\ensuremath{#3}}
}
\newcommand{\stackinput}{%
	\psline[linecolor=darkgray]{c->}(24, 17.5)(16, 17.5)(16, 14)
}
\newcommand{\stackshortinput}{%
	\psline[linecolor=darkgray]{c->}(22.5, 17.5)(16, 17.5)(16, 14)
}
\newcommand{\stackoutput}{%
	\psline[linecolor=darkgray]{c->}(16, 14)(16, 17.5)(10, 17.5)
}
\newcommand{\stackinoutput}{%
	\psline[linecolor=darkgray]{c->}(24, 17.5)(17, 17.5)(17, 13)
	\psline[linecolor=darkgray]{c->}(15,14)(15, 17.5)(10, 17.5)
}
\newcommand{\firstpass}{%
	\rput[c](16, -.5){\text{First pass through the stack.}}
}
\newcommand{\secondpass}{%
	\rput[c](16, -.5){\text{Second pass through the stack.}}
}
\newcommand{\thirdpass}{%
	\rput[c](16, -.5){\text{Third pass through the stack.}}
}

\begin{abstract}
  We consider the avoidance of patterns in inversion sequences that relate sorting via sorting machines including data structures such as stacks and pop stacks.  Such machines have been studied under a variety of additional constraints and generalizations, some of which we apply here.  We give the classification of several classes of sortable inversion sequences in terms of pattern avoidance.  We are able to provide an exact enumeration of some of the sortable classes in question using both classical approaches and a more recent strategy utilizing generating trees.  
\end{abstract}

\section{Introduction}
An \emph{inversion} in a permutation $\pi =\pi_1\pi_2\pi_3\cdots \pi_n$ is a pair of entries $\pi_i,\pi_j$ where $i<j$, but $\pi_i >\pi_j$.  An \emph{inversion sequence}  of length $n$ is a word $e=e_1\ldots e_{n}$ which satisfies for each
$i \in \{1,\ldots,n\}$ the inequalities $0\le e_i\leq i-1$. These sequences correspond to permutations via a variant of the Lehmer code~\cite{laisant:sur-la-numerati:, lehmer:teaching-combin:} by mapping each entry $\pi_j$ of the permutation $\pi$ to $e_{j-1}$ where $e_{j-1}$ is the number of inversions in $\pi$ where $\pi_j$ is the second, smaller entry of the pair.  The set of inversion sequences (respectively the set of inversion sequences of length $n$) is
denoted $\I$ (respectively $\I_n$).  

\begin{example} The permutation $\pi = 51743862$ corresponds to the inversion sequence $e = 01023026$.
\end{example}

The study of pattern avoidance in inversion sequences began with the work of Corteel, Martinez, Savage, and Weselcouch~\citeyear{corteel2016} as well as that of  Shattuck and the first author~\citeyear{inversion2015}.  Since then, extensive work has been done in this area, including~\cite{Lin,LinKim,YanLin} to list a few.  Recently, Kotsireas,  Y\i ld\i r\i m, and the first author~\citeyear{MY} introduced an algorithmic technique involving generating trees to enumerate many pattern classes of inversion sequences (also see \cite{CJM,CM23} and the references therein).  We utilize this algorithm in Section~\ref{S:GenTree} to obtain a generating function for the number of inversion sequences sortable by two particular machines and give the succession rules to build the generating tree for a third machine.

There remain many open problems in characterizing and enumerating the inversion sequences that avoid a class consisting of a small number of relatively short (length less than five) patterns.  Our interest is in the characterization and enumeration of inversion sequences sortable by stacks and pop stacks.  We also study the characterization of inversion sequences sortable by data structures obtained by generalizing or further restricting the pop stacks.

A stack is a last-in, first-out sorting device with push and pop operations.  Knuth~\citeyear{knuth:the-art-of-comp:1} showed the permutation $\pi$ can be sorted by a stack (that is, by applying push and pop operations to the sequence $\pi_1,\dots,\pi_n$ one can output the identity permutation $1,\dots,n$) if and only if $\pi$ avoids the permutation $231$.

\begin{definition}  A permutation $\pi=\pi_1\pi_2\dots\pi_n\in S_n$ is said to \emph{contain} a permutation $\sigma=\sigma_1\sigma_2\ldots\sigma_k$ if there exist indices $1\le \alpha_1<\alpha_2< \ldots <\alpha_k \le n$ such that $\pi_{\alpha_i} < \pi_{\alpha_j}$ if and only if $\sigma_i <\sigma_j$.  Otherwise, we say $\pi$ \emph{avoids} $\sigma$.  

These definitions of \emph{containment} and \emph{avoidance} have been naturally extended to words (and specifically inversion sequences) with the added allowance of entries being equal to one another.
\end{definition}

\begin{example}  The permutation $\pi = 241563$ contains $312$ since the $4,1,3$ appear in the same relative order as $3,1,2$.   However, $\pi$ avoids $321$ since there is no decreasing subsequence of length three in $\pi$.

Similarly, the inversion sequence $\tau = 0021104$ contains $1002$ from $2,1,1,4$, but avoids $201$.
\end{example}

Describing the permutations that are sortable by machines consisting of one or more data structures (such as stacks or pop stacks) is a point of interest in the field.  To this end, it is often possible to describe the sortable machine in terms of its \emph{basis}.

\begin{definition}
A \emph{permutation class} is a downset of permutations under the containment order.  Every permutation class can be specified by the set of minimal permutations which are \emph{not} in the class called its \emph{basis}.  For a set $B$ of permutations, we denote by $\Av(B)$ the class of permutations which do not contain any element of $B$.
\end{definition}

In the same vein, for any set of words $R$, denote by $\I(R)$ the set of inversion sequences in $\I$ that avoid every word in $R$.  Similarly, denote by $\I_n(R)$ the set of inversion sequences of length $n$ that avoid every word in $R$.

As mentioned above, the stack-sortable permutations are the class $\Av(231)$.

A restricted version of a stack that has received a fair amount of attention is a pop stack, first introduced by Avis and Newborn~\citeyear{avis:on-pop-stacks-i:}.  A pop stack supports the same push and pop operations as a standard stack, but whenever a pop operation occurs, every entry in the pop stack is popped immediately.  One foundational result is that the permutations that can be sorted by a single pop stack are exactly those that avoid $231$ and $312$.  Some later works involving pop stacks include~\cite{ABBHL2019, atkinson:pop-stacks-in-p:, Cerbai, CG2019, DW2022,  EG2021, H2022, PS2019,  smith:the-enumeration:}.

Atkinson~\citeyear{atkinson:generalized-sta:} explored a variation of a stack called an $(r,s)$-stack which extends the traditional push and pop operations.  Specifically, the first index of an $(r,s)$ stack allowed one to push an entry not only to the top of the stack but to any of the top $r$ spots.  Similarly, the second index expanded the options of the pop operation to allow one to pop any of the top $s$ entries from the stack.  Thus a $(1,1)$-stack is simply a traditional stack.  Atkinson was able to characterize these permutations when one of $r,s$ was fixed at $1$, as well as in the case when $r=s=2$.  Natural variations on stacks arise when studying inversion sequences that are similar in flavor to these generalized stacks.

Elder~\citeyear{elder:permutations-ge:} considered restricting the depth of a stack, that is the maximum number of entries allowed in a stack at any given time, in the context of generating permutations with two stacks in series.  Elder later continued this study with Lee and Rechnitzer~\citeyear{ELR2016} and then extended to sorting and pop stack sorting with Goh~\citeyear{EG2018, EG2021}.  Although this work focused on a permutation sorting (or generating) machine consisting of two stacks in series, we will consider a similar restriction for a single stack or pop stack when sorting inversion sequences.

Note that natural sorting machines consisting of stacks and/or pop stacks do not give any advantage in sorting a sequence by inserting extra entries that might force a stack or pop stack to be popped at an earlier stage.   As such, the sortable words from any such machine form a class and can be completely characterized by classical pattern avoidance. 

\section{Stack sorting inversion sequences}

We give the standard optimal stack sorting algorithm extended to words in Algorithm~\ref{A:Stack}.  Note that a word is considered sorted if all of its entries appear in weakly increasing order.

\begin{algorithm}
\caption{Stack Sorting Algorithm}\label{A:Stack}
\begin{algorithmic}
\State $w=w_1w_2\cdots w_n$ is a word of length $[n]$.
\State $S$ is a stack.
\State $O$ is an empty array.
\State $i = 1$.
\Procedure{StackSort}{$w,n$}
\While{$i < n+1$}
	\If{$S$ is empty}
		\State push $w_i$ to the top of $S$
		\State $i = i+1$  
	\ElsIf{$w_i \leq $ top element of $S$}
   		 \State push $w_i$ to the top of $S$
		 \State $i = i+1$
   	\Else
		\State pop the top entry of $S$ to the end of $O$
	\EndIf
\EndWhile \\
\Return $O$
\EndProcedure
\end{algorithmic}
\end{algorithm}

For a word $w$ to be stack sortable, $w$ must avoid $120$, the same as avoiding $231$ which is required for permutations to be sortable.  The repetition allowed in words does not change this requirement from being both necessary and sufficient and so the proof does not change either.  It is included below for completeness.

\begin{proposition}~\label{P:stackwords}
The words that are stack sortable are exactly the words that avoid $120$.
\end{proposition}

\begin{proof}
The $2$ of a $120$ pattern forces the smaller entry $1$ to the output before the $0$ even enters the stack which creates an inversion with the $1,0$.  Hence words containing a $120$ pattern are not stack sortable.

Next consider a word $w$ that is not stack sortable.  The output  $S(w)$ must contain an inversion, say an entry $b>a$ that appears before $a$.  For $a$ to not have exited the stack before $b$, we know $a$ appears after $b$ in $w$.  Furthermore, for $b$ to have been forced out of the stack before $a$ entered, there must have been a larger entry $c$ that appeared after $b$, but before $a$ to force $b$ out before $a$ could enter the stack.  Thus $w$ contains a $120$ pattern $bca$.
\end{proof}

As inversion sequences are words with added restrictions, Proposition~\ref{P:stackwords} also gives us a basis for the stack sortable inversion sequences.

\begin{corollary}
The inversion sequences that are stack sortable are exactly those that avoid $120$.
\end{corollary}

The enumeration of these sortable inversion sequences has been studied, but thus far remains open.  See \cite[Section 4]{inversion2015} and \cite{corteel2016} for initial study and further work on avoiding $120$ in~\cite{YanLin}.

\section{Pop stack sorting inversion sequences}

The optimal pop stack sorting algorithm is mostly the same procedure as that for stacks, but when a pop stack is popped, all entries must exit the stack as described in Algorithm~\ref{A:PStack}.  Note that the end result of this algorithm is a reversal of each maximal weakly descending run of the input sequence.

\begin{algorithm}
\caption{Pop Stack Sorting Algorithm}\label{A:PStack}
\begin{algorithmic}
\State $w=w_1w_2\cdots w_n$ is a word of length $[n]$.
\State $PS$ is a pop stack.
\State $O$ is an empty array.
\State $i = 1$.
\Procedure{PopStackSort}{$w,n$}
\While{$i < n+1$}
	\If{$PS$ is empty}
		\State push $w_i$ to the top of $PS$
		\State $i = i+1$  
	\ElsIf{$w_i \leq $ top element of $PS$}
   		 \State push $w_i$ to the top of $PS$
		 \State $i = i+1$
   	\Else
		\While{$PS$ is nonempty}
			\State pop the top entry of $PS$ to the end of $O$
		\EndWhile
	\EndIf
\EndWhile \\
\Return $O$
\EndProcedure
\end{algorithmic}
\end{algorithm}

Note that $120, 201$ are the patterns whose avoidance completely determines the pop stack sortability of a permutation.  However, due to the repetition allowed in words, a copy of a larger element already in the stack can force the entire stack to be popped while there is still a copy of a smaller entry appearing later in the permutation.  

We note that  Cerbai~\citeyear[Theorem 5.2]{Cerbai}  found the basis for sortable Cayley words in the context of a \emph{hare pop stack} which is simply a pop stack that only allows entries to appear in weakly decreasing order from top to bottom (which is a necessary requirement for sorting with a single stack). This basis is the same as the one in Theorem~\ref{T:popwords} below.  Cerbai's argument could also be extended to sorting general words by a pop stack as we have done here.  

\begin{theorem}~\label{T:popwords}
The words that are pop stack sortable are exactly the words that avoid $120,201,$ and $1010$.
\end{theorem}

\begin{proof}
One can check that none of $120,201,$ and $1010$ are pop stack sortable.  As such, a word that contains any of these patterns cannot be sorted by a pop stack.

Conversely, suppose $w$ is not pop stack sortable.  Then the image of $w$ under the pop stack operator, which we will denote by $PS(w)$, must contain an inversion $(y,x)$.  That is, $y> x$ and $y$ appears before $x$ in $PS(w)$.  Consider what happens to force the entry $y$ to the output before $x$.  

Using Algorithm~\ref{A:PStack}, there must be a point when $y$ is in the stack and an entry $d$ that appears (immediately) after a smaller entry $b$ in $w$. Because $b<d$ and $b$ appears on the top of the pop stack when $d$ is the next entry of $w$, the pop stack must be popped before $d$ can enter.  As $x$ appears after $y$ in $PS(w)$, it must be that $x$ is still in the input at this state.  These entries (if they are all distinct) are shown before entering the pop stack (left) and in the state where $b$ just entered the pop stack (right) in Figure~\ref{F:popstack_fail}.

If $y=b$, then $x$ appears after $d$ in $w$.  Thus $bdx$ forms a $120$ pattern in $w$.  

Otherwise, there exists an entry $y>b$ that is also forced out before $d$ enters the pop stack but before $x$.  Thus, $y$ appears before $b$ in $w$.  Either we can assume $x=d$ (that is, $y>d$) or if $x<d$, then $x$ appears after $d$ in $w$.  In the former case, $ybd$ is a $201$ pattern in $w$.  In the latter case, $x<d$ appears after $d$, so consider the pattern $ybdx$ in $w$ where we are reduced to the options $d \geq y > x ,b$.  Thus, $ybdx$ is one of the following patterns $1010, 1020, 2021, 2120, 2031, 2130$.  Of these patterns, only $1010$ does not contain one of the smaller forbidden patterns $120, 201$.
\end{proof}

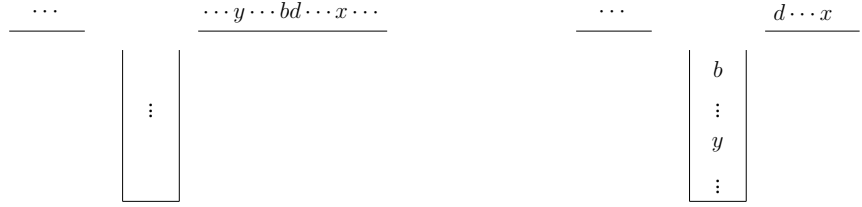
\begin{figure}
\begin{center}
\scalebox{0.5}{\begin{tikzpicture}
	\begin{scope}
					\draw (3.5,1) -- (5,1);
					\draw (3.5,1) -- (3.5,5);
					\draw (5,1) -- (5,5);
					\draw (5.5, 5.5) -- (10.5,5.5);
					\draw (0.5, 5.5) -- (2.5,5.5);
					\node[][align=left] at (1.5,6) {\fontsize{18}{22.4}\selectfont$\cdots $};
					\node[][align=left] at (8,6) {\fontsize{18}{22.4}\selectfont$\cdots y \cdots bd \cdots x \cdots$};
					\node[][align=left] at (4.25,3.5) {\fontsize{18}{22.4}\selectfont$\vdots$};
	\end{scope}
	
	\begin{scope}[xshift = 15cm]
					\draw (3.5,1) -- (5,1);
					\draw (3.5,1) -- (3.5,5);
					\draw (5,1) -- (5,5);
					\draw (5.5, 5.5) -- (8,5.5);
					\draw (0.5, 5.5) -- (2.5,5.5);
					\node[][align=left] at (1.5,6) {\fontsize{18}{22.4}\selectfont$\cdots $};
					\node[][align=left] at (6.5,6) {\fontsize{18}{22.4}\selectfont$d \cdots x$};
					\node[][align=left] at (4.25,1.5) {\fontsize{18}{22.4}\selectfont$\vdots$};
					\node[][align=left] at (4.25,2.5) {\fontsize{18}{22.4}\selectfont$y$};
					\node[][align=left] at (4.25,3.5) {\fontsize{18}{22.4}\selectfont$\vdots$};
					\node[][align=left] at (4.25,4.5) {\fontsize{18}{22.4}\selectfont$b$};
	\end{scope}

\end{tikzpicture}}
\end{center}
\caption{Moving to the state where $y$ will be forced to exit the pop stack before $x$ enters}
\label{F:popstack_fail}
\end{figure}

As inversion sequences are words with added restrictions, Theorem~\ref{T:popwords} also gives us a basis for the pop stack sortable inversion sequences.

\begin{corollary}
The inversion sequences that are pop stack sortable are those in  $\I(120, 201, 1010)$.
\end{corollary}

The pop stack sortable permutations are also known as \emph{layered} permutations which are permutations of the form $\pi = \sigma_1 \sigma_2 \cdots \sigma_k$ where each \emph{layer} $\sigma_i$ is a maximal length contiguous decreasing sequence such that all of the entries of $\sigma_i$ are less than all of the entries of $\sigma_{i+1}$.  For example, the permutation $\pi = 543216987$ is a layered permutation with layers $\sigma_1=54321, \sigma_2=6, \sigma_3=987$ as shown on the left in Figure~\ref{F:layered}.

\begin{figure}
\begin{center}
\scalebox{0.3}{\begin{tikzpicture}
					\draw (0,0) rectangle (9,9);
					\draw (0,0) rectangle (5,5);		
					\fill[black] (.5,4.5) circle (.2cm);
					\fill[black] (1.5,3.5) circle (.2cm);
					\fill[black] (2.5,2.5) circle (.2cm);
					\fill[black] (3.5,1.5) circle (.2cm);
					\fill[black] (4.5,0.5) circle (.2cm);
					\draw (5,5) rectangle (6,6);
					\fill[black] (5.5,5.5) circle (.2cm);
					\draw (6,6) rectangle (9,9);
					\fill[black] (6.5,8.5) circle (.2cm);
					\fill[black] (7.5,7.5) circle (.2cm);
					\fill[black] (8.5,6.5) circle (.2cm);
					\draw (4,-1) node[anchor =north]  {\Huge $\pi = 543216987$};
					
					\draw (20,0) rectangle (29,9);
					\draw (20,0) rectangle (24,2.5);		
					\fill[black] (20.5,2.5) circle (.2cm);
					\fill[black] (21.5,2.5) circle (.2cm);
					\fill[black] (22.5,1.5) circle (.2cm);
					\fill[black] (23.5,0.5) circle (.2cm);
					\draw (24,2.5) rectangle (27,3.5);
					\fill[black] (24.5,3.5) circle (.2cm);
					\fill[black] (25.5,3.5) circle (.2cm);
					\draw (27,3.5) rectangle (29,9);
					\fill[black] (26.5,2.5) circle (.2cm);
					\fill[black] (27.5,7.5) circle (.2cm);
					\fill[black] (28.5,5.5) circle (.2cm);
					\draw (24,-1) node[anchor =north]  {\Huge $w=211033275$};
\end{tikzpicture}}
\end{center}
\caption{The layered permutation $\pi=543216987$ and the layered word $w=211033275$.}
\label{F:layered}
\end{figure}
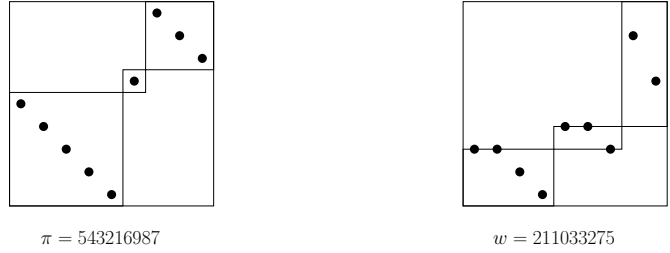

We can define layered words similarly as follows.

\begin{definition}  A \emph{layered} word $w$ is a word $w = \tau_1\tau_2\cdots \tau_k$ where each $\tau_i$ is
\begin{enumerate}
\item maximal in length,
\item weakly decreasing, and
\item such that the last (weakly smallest) entry of $\tau_{i+1}$ is at least as large as the first (weakly largest) entry of $\tau_i$ for all $i$.
\end{enumerate}
\end{definition}

For example, the word $w = 211033275$ is a layered word with layers $2110, 332,75$ shown on the right in Figure~\ref{F:layered}.

 It can then be seen that these layered words are exactly the pop stack sortable words.

\begin{proposition}  Layered words are exactly the words that avoid $120, 201,1010$.
\end{proposition}

\begin{proof}  Suppose $w$ is layered.  Then $w$ cannot contain $120$ because if not, the layer containing the $0$ of the pattern must appear after the layer containing the $1$ which violates the third condition.  The same argument can be made for the $1$ of the $201$ pattern appearing in a later consecutive decreasing pattern than that of the $2$ and once again for the second $0$ of the $1010$ pattern appearing in a later consecutive decreasing pattern than that of the first $1$.

Now suppose $w$ avoids $120, 201,1010$.  Then $w$ is pop stack sortable.  Consider the output from sorting $w$ broken into segments based on the pops.  Reversing those segments reverts the word back to $w$.  In addition, according to the algorithm, these breaks occur exactly at the strict ascents in $w$.  The segments then are weakly decreasing so that the smallest entry of any segment is at least as large as the largest entry of the previous segment for $w$ to be sortable and also maximal in length.  Hence $w$ is layered.
\end{proof}

In Section~\ref{S:popgen}, we are able to describe the generating tree for layered inversion sequences.  There are also some partial results we can obtain combinatorially based on the number of layers (which is also the number of pops needed to sort the sequence).  Recall that the Eulerian numbers $E(n,k)$ count the number of permutations of length $n$ with $k-1$ descents.

\begin{proposition}  The number of pop stack sortable inversion sequences of length $n$ with $k=1,2$ layers are counted by $E(n,k)$.
\end{proposition}

\begin{proof}  In fact, the pop sortable inversion sequences of length $n$ with $k=1,2$ layers correspond exactly to the permutations of length $n$ with $k-1$ descents by the variant of Lehmer code used to create inversion sequences.  In the trivial case of $k=1$, a permutation of length $n$ with $k-1=0$ descents is the identity permutation which is the only permutation of length $n$ to have an inversion sequence with one layer, namely the inversion sequence consisting of all $0$s.

Now consider the case when $k=2$.  A permutation $\pi$ with exactly one descent, say at index $i$, will have an inversion sequence that begins with $i$ 0s which will be the first layer.  Since the permutation $\pi$ is such that $\pi_{i+1}$ is the smaller entry in an inversion with at least $\pi_i$, the corresponding inversion sequence now has a positive entry in position $i+1$.  As the remainder of the entries of $\pi$ are all in increasing order, each subsequent entry of $\pi$ can be the smaller entry of at most as many inversions as the previous entry.  Thus, the rest of the inversion sequence is in weakly decreasing order meaning that the entries from position $i+1$ to position $n$ form the second and final layer.
\end{proof}

As not all inversion sequences are pop stack sortable, this correspondence can clearly not continue indefinitely.  Indeed, a permutation with even $k-1 =2$ descents is not guaranteed a corresponding pop stack sortable inversion sequence.  For example, the permutation $\pi = 3214$ has the inversion sequence $e(\pi) = 0120$, which cannot be sorted by a pop stack.

At the other extreme, only one inversion sequence of any given length $n \geq 1$ can have $n$ layers, namely the increasing sequence $012\ldots(n-1)$.  Reducing the number of layers by one to $n-1$ layers gives a set of inversion sequences enumerated by the Tetrahedral numbers, sequence number A000292 in OEIS~\cite{OEIS}.

\begin{proposition}  The number of pop stack sortable inversion sequences of length $n \geq 1$ with $n-1$ layers is $\binom{n+1}{3}$.
\end{proposition}

\begin{proof}  A pop stack sortable inversion sequence $e$ of length $n \geq 1$ with $n-1$ layers will have exactly one entry $e_i$ as the second entry of a layer where  $i \in \{2,3,4,\ldots,n\}$.  It must be the case that $e_m =m-1$ for all $m<i$ for those entries to have all started new layers.  Note that $e_i \geq e_{i-2}$ as otherwise $e_{i-2}, e_{i-1}, e_i$ would form a forbidden $120$ pattern in $e$.  This means $e_i = e_{i-1} = i-2$ or $e_i=e_{i-2}=i-3$.  Further, we know $e_i, e_{i+1},e_{i+2},\ldots,e_n$ are in strictly increasing order as they are all in different layers.

In the case that $e_i = i-2$, since we must have $i-1 \leq e_{i+1} < e_{i+2} < \cdots < e_n \leq n-1$, we can select up to one of these later entries to be two greater than the previous entry.  Consider positions $\{2,3,4,\ldots,n,n+1\}$ where we pick two positions; the first to be the position of the second entry in the layer of length two and the second to be the position of the entry that is two more (as opposed to one more) than the previous.  In the case that $n+1$ is one of the selected options, then we do not have any entry that is two more than the one before it.  Hence there are $\binom{n}{2}$ such inversion sequences.

Otherwise, $e_i = i-3$ and so $i \geq 3$.  Now because $i-2 \leq e_{i+1} < e_{i+2} < \cdots < e_n \leq n-1$, up to two of these later entries could be two greater than the previous entry or one of these later entries could be three greater than the previous entry.  For the cases where no entry is three more than the previous entry, we have $\binom{n-2}{1} + \binom{n-2}{2} + \binom{n-2}{3}$ ways to select the second entry of the length two layer and the entries that are two greater than the entry they follow.  In the case where we have one entry that is three more than the previous entry, there are $\binom{n-2}{2}$ ways to select this entry and the entry that was the second in the longest layer.   This gives us a total of $\binom{n-2}{1} + 2\binom{n-2}{2} + \binom{n-2}{3} =\binom{n}{3}$ via a generalization of Pascal's Identity.

By another application of Pascal's Identity, there are $\binom{n}{2} + \binom{n}{3} = \binom{n+1}{3}$ pop stack sortable inversion sequences of length $n \geq 1$ with $n-1$ layers.
\end{proof}

One last special case is when the sortable inversion sequence contains only $0$s and $1$s and thus only needs to avoid $1010$, thereby having at most three layers.  The number of such sequences is the number of compositions of $n$ into at most four parts because we must begin with a $0$ and can only alternate values at most three times after that.  These numbers (with an offset of one) are better known as the ``Cake numbers"~\cite{YY}, sequence number A000125 in OEIS~\cite{OEIS}, which count the maximum number of slices of a cylindrical cake one can have after making $n-1$ planar cuts perpendicular to the top circular surface.

\begin{proposition}
The number of pop stack sortable inversion sequences of length $n$ only containing $0,1$ is the $(n-1)$st Cake number: $\binom{n-1}{0}+\binom{n-1}{1}+\binom{n-1}{2}+\binom{n-1}{3}$.
\end{proposition}

We note that Cerbai~\citeyear{Cerbai} also considered a \emph{tortoise pop stack} for sorting Cayley words that had an additional pattern restriction of not allowing a $00$ pattern to appear in the pop stack.  (Restricting the content of stacks in terms of pattern avoidance was initially considered in~\cite{AMR, RS, CCF}.)  Cerbai's characterization of sortable Cayley words can be easily extended to all words, so we omit the proof here.

\begin{proposition}
The words (and thus inversion sequences) that are tortoise pop stack sortable are those which avoid $120, 201, 110,$ and $100$.
\end{proposition}

Given this classification, the generating function for the enumeration of tortoise pop stack sortable inversion sequences is given in Theorem 12 in recent work by Callan and Mansour~\citeyear{CM23}.

\subsection{Restricting pop stack depth}

When extending the notion of the depth of a stack to words, we could continue to use the definition in terms of the number of elements allowed in the stack as introduced by Elder~\citeyear{elder:permutations-ge:}.  This could be more practical in some applications.  However, given the repetition that distinguishes words from permutations and considering the storage in terms of distinct elements to remember, it also makes sense to consider another definition that is still consistent with the original definition when applied to permutations.  The content of the stack could be stored as a limited number (the depth) of lists of positions in the stack where these positions are counted from the bottom.  We consider the latter notion as given in the following definition.

\begin{definition} The \emph{depth} of a stack (or pop stack) is the number of entries with distinct values allowed in the stack (or pop stack) at any stage.
\end{definition}

For example, a stack of depth three under this definition would still only allow three entries of a permutation to be in the stack at one time.  However, this same restricted stack could allow arbitrarily many entries of a word in the stack, provided that no more than three distinct letters were represented at any time.  Using this definition, we begin by characterizing words and specifically inversion sequences sortable by pop stacks of limited depth.

\begin{proposition}  A word $w$ is sortable by a pop stack of depth $k$ if and only if $w$ avoids $120,201,1010$ and also $k(k-1)\cdots10$.
\end{proposition}

\begin{proof}  As before any word $w$ must avoid $120,201,1010$ to be pop stack sortable.  Additionally, a decreasing sequence that is too long to fit in a limited depth pop stack is also not sortable.

Conversely, if $w$ is not pop stack sortable by a pop stack of depth $k$ and avoids $120,201,1010$, then the problem is a result of the limitation of the pop stack's capacity.  Hence $w$ must contain the $k(k-1)\cdots10$ pattern.
\end{proof}

\begin{corollary}~\label{C:popstackD} Inversion sequences sortable by a pop stack of depth $k$ are exactly those in  \\ $\I(120,201,1010, k(k-1)\cdots10)$.
\end{corollary}

In the case of depth one, the number of sortable inversion sequences is the number of weakly increasing inversion sequences which are known to be enumerated by the Catalan numbers due to Martinez and Savage~\citeyear{MS2018}.  The proof shows these inversion sequences to correspond to the $213$ avoiding permutations, which were enumerated by Simion and Schmidt~\citeyear{simion:restricted-perm:}.

\begin{theorem}  (Martinez and Savage) The number of inversion sequences of length $n$ which avoid $10$ and thus are sortable by a pop stack of depth $1$ is the $n$th Catalan number $C_n =\frac{\binom{2n}{n}}{n+1}$.
\end{theorem}

For a pop stack of depth two, the sortable inversion sequences would need to avoid $120,201,210, 1010$.  The first few terms of the enumeration sequence are: $1,1,2,6,23,100,471,2349$.  We define the succession rules for the generating tree for these inversion sequences (using the algorithm introduced in \cite{MY}) in Section~\ref{S:ps_d2}.  However, it is hard to get an explicit formula by translating this generating tree to a system of equations and it is even more difficult for larger values of $k$.

Another approach is to count the inversion sequences recursively directly, but this too is not as nice as we would hope.  One recursive characterization is given below.

A weakly decreasing word of length $n$ has depth $k$ if it has the form $a_1\ldots a_1a_2\ldots a_2\ldots a_{k+1}\ldots a_{k+1}$ where for each $i$, $a_i > a_{i + 1}\ge 0$ (so there are $k$ descents). Let $WD_k(n,a)$ be the set of all such words whose first value is $a$ and whose depth is at most $k$.

\begin{lemma} For all $n,k\in\mathbb{Z}^+$ with $k<n$, we have
\begin{equation}
|WD_k(n,a)| = 1 + \sum_{j=1}^{k}\binom{n-1}{j}\binom{a}{j}.
\end{equation}
\end{lemma}

Note that we assume $\binom{n}{k}=0$ if $k>n$. The result follows from choosing $j$ locations for a descent, then $j$ values for the weakly decreasing sequence, and one for the sequence with no descents.

We say a word $w$ of length $n$ is a layered weakly decreasing word of depth at most $k$ if $w$ can be split into a sequence of subwords (in the factor sense) $w_1,w_2,...$ where each $w_i \in WD_k(m_i,a_i)$ with $\sum_i m_i=n$, and for each $i$, $a_i\ge 0$, and $min(w_{i + 1})\ge max(w_i)$ (i.e. the largest (first) value of each word is bounded by the smallest value of the succeeding word). Further, we consider the words of this type which are also inversion sequences, i.e. $max(w_i)\le \sum_{j = 1}^{i - 1}m_j$. Let $LDI_k(n)$ be the set of all such layered depth at most $k$ inversion sequences of length $n$.

\begin{lemma} The set of sortable inversion sequences by a depth $k$ pop stack is exactly $LDI_k(n)$.
\end{lemma}

The proof is simply to note that when the stack is popped, there are at most $k$ distinct values, and they must be less than any later sequence values.

Let $w_1w_2...w_t \in LDI_k(n)$, where $w_1,w_2,...,w_t$ are the subwords as defined above with each value of $w_{i+1}$ bounded below by the largest value in $w_i$.  Subtracting $a_i$ from every value in $w_{i+1}$ yields an element of $WD_k(m_i, a_{i+1}-a_i)$.

As a result we get the following formula.

\begin{proposition}  For all $n,k>0$,
\begin{equation}~\label{E1}
 |LDI_k(n)| = \sum_{\substack{m_1+m_2+...m_t=n,\\m_i\ge 0}} \quad \sum_{a_1,a_2,...a_t}^* |WD_k(m_i, a_i)|,
\end{equation}
 where $a_1=0$, and $\sum^*$ means for all $j > 1$ to sum on values of $a_j$ such that $\displaystyle{0\le a_j\le \bigg(\sum_{h = 1}^{j - 1}m_h - a_h\bigg) }$.
\end{proposition}

Since the first word in an element of $LDI_k(n)$ must be a string of zeros, we let $LDI_k(n,m_1)$ denote the elements of $LDI_k(n)$ that begin with $m_1$ zeros followed by a non-zero value. Further, let $LDI_k(n, m_1, a_2)$ be the subset of these layered decreasing inversion sequences which begin with $m_1$ zeros, and have first non-zero value $a_2$ (and $a_2 \le m_1$). We now decompose our sequence of words as $\{\overline{0}_{m_1},w_2,\hat{w}\}$ where $\hat{w} = \{w_3,...w_t\}$. Now $w_2 \in WD_k(m_2, a_2)$ for some $0\le m_2 \le n-m_1$, $a_2>0$. If we subtract $a_2$ from all the elements in $\hat{w}$ the result is in $LDI(n - m_2 - m_1, m_1 + m_2 - a_2, a)$  for some $a$ with $a\le m_1 + m_2 - a_2$. Note that the number of starting zeros guarantees that adding $a_2$ to all of these entries will still yield an element of $LDI_k(n)$.

Hence we get the following formula.

\begin{proposition} For all $n,k,m_1\in\mathbb{Z}^+$ with $m_1<n$,
\begin{equation}~\label{E2} |LDI_k(n,m_1)| = \sum_{m_2 = 0}^{n-m_1} \sum_{a_2\le m_1}|WD_k(m_2,a_2)|\sum_{a\le m_1 + m_2 - a_2} |LDI_k(n - m_1 - m_2, m_1 + m_2 - a_2, a)|.
\end{equation}
\end{proposition}

Equations~\ref{E1} and \ref{E2} are rather difficult to use for computation. However they are presented to give some insight into directly enumerating these inversion sequences.

For completeness, we give some characterizations of words sortable by stacks of depth $k$.

\begin{proposition}  A word $w$ is sortable by a stack of depth $k$ if and only if $w$ avoids $120$ and also $k(k-1)\cdots10$.
\end{proposition}

\begin{corollary} The number of words of length $n$ on an alphabet $[k]$ sortable by a stack of depth $1$, i.e. avoiding $10$ is $\binom{n+k-1}{k-1}$.
\end{corollary}

\begin{proof}  These weakly increasing words can be thought of as weak compositions of $n$ into $k$ parts where each part represents the next largest letter.
\end{proof}

The depth 2 stack sortable words are enumerated by Burstein~\citeyear{BB} in the context of avoiding $120, 210$.

\begin{theorem} (Burstein)  The number of words of length $n$ on an alphabet $[k]$ avoiding $120, 210$ and thus sortable by a stack of depth $2$ is given by $$\frac{1 - (-1)^k}{2} + 2^n\sum_{i=0}^{\lfloor \frac{k-2}{2} \rfloor} \binom{n+k-3-2i}{n-1}.$$
\end{theorem}

Finally, we note that Kotsireas, Y\i ld\i r\i m, and the first author~\citeyear{MY} give a functional equation for the inversion sequences of length $n$ avoiding $120, 210$ and thereby sortable by a stack of depth $2$.

\subsection{Generalizing pop stacks}

We can also expand the ability of our sorting machine using a similarly defined version of Atkinson's~\citeyear{atkinson:generalized-sta:} extension of push and pop operations.  Specifically, we will define a  $(r,1)$-pop stack  that is also modified to take into account the repetition of letters in inversion sequences (and more generally words).   In Atkinson's work, an $(r,1)$-stack would have allowed for push operations to push entries into the stack in any of the top $r$ positions.  We can again naturally redefine generalized stacks for words in a way that is consistent with Atkinson's work on permutations.

\begin{definition}  Let an  $(r,1)$-stack allow the push operation to push an entry to any of the top $r$ ``value positions'' in the stack where consecutive entries of the same value in the stack are grouped together.  
\end{definition}

Note that the above definition could be extended to the broader notion of an $(r,s)$ stack where we allow the pop operation to be defined in terms of these same value positions.

\begin{example} Suppose a $(3,1)$-stack contained the values $0,0, 3, 4,6, 6$ as shown in Figure~\ref{F:(3,1)-stack} with the remainder of the word being $230$. The available positions for new entries to enter the stack at each stage are indicated by the $\square$s.  
\end{example}

\begin{figure}
\scalebox{0.42}{\begin{tikzpicture}
	\begin{scope}
					\draw (3.5,1) -- (5,1);
					\draw (3.5,1) -- (3.5,8);
					\draw (5,1) -- (5,8);
					\draw (5.5, 8.5) -- (10,8.5);
					\node[][align=left] at (6,9) {\fontsize{18}{22.4}\selectfont$230$};
					\node[][align=left] at (4.25,1.5) {\fontsize{18}{22.4}\selectfont$6$};
					\node[][align=left] at (4.25,2.1) {\fontsize{18}{22.4}\selectfont$6$};
					\node[][align=left] at (4.25,2.7) {\fontsize{18}{22.4}\selectfont$4$};
					\node[][align=left] at (4.25,3.3) {\fontsize{18}{22.4}\selectfont$\square$};
					\node[][align=left] at (4.25,3.9) {\fontsize{18}{22.4}\selectfont$3$};
					\node[][align=left] at (4.25,4.5) {\fontsize{18}{22.4}\selectfont$\square$};
					\node[][align=left] at (4.25,5.1) {\fontsize{18}{22.4}\selectfont$0$};
					\node[][align=left] at (4.25,5.7) {\fontsize{18}{22.4}\selectfont$0$};
					\node[][align=left] at (4.25,6.3) {\fontsize{18}{22.4}\selectfont$\square$};
	\end{scope}	
	
	\begin{scope}[xshift = 9cm]
					\draw (3.5,1) -- (5,1);
					\draw (3.5,1) -- (3.5,8);
					\draw (5,1) -- (5,8);
					\draw (5.5, 8.5) -- (10,8.5);
					\node[][align=left] at (6,9) {\fontsize{18}{22.4}\selectfont$30$};
					\node[][align=left] at (4.25,1.5) {\fontsize{18}{22.4}\selectfont$6$};
					\node[][align=left] at (4.25,2.1) {\fontsize{18}{22.4}\selectfont$6$};
					\node[][align=left] at (4.25,2.7) {\fontsize{18}{22.4}\selectfont$4$};
					\node[][align=left] at (4.25,3.3) {\fontsize{18}{22.4}\selectfont$3$};
					\node[][align=left] at (4.25,3.9) {\fontsize{18}{22.4}\selectfont$\square$};
					\node[][align=left] at (4.25,4.5) {\fontsize{18}{22.4}\selectfont$2$};
					\node[][align=left] at (4.25,5.1) {\fontsize{18}{22.4}\selectfont$\square$};
					\node[][align=left] at (4.25,5.7) {\fontsize{18}{22.4}\selectfont$0$};
					\node[][align=left] at (4.25,6.3) {\fontsize{18}{22.4}\selectfont$0$};
					\node[][align=left] at (4.25,6.9) {\fontsize{18}{22.4}\selectfont$\square$};
	\end{scope}	
	
	\begin{scope}[xshift = 18cm]
					\draw (3.5,1) -- (5,1);
					\draw (3.5,1) -- (3.5,8);
					\draw (5,1) -- (5,8);
					\draw (5.5, 8.5) -- (10,8.5);
					\node[][align=left] at (6,9) {\fontsize{18}{22.4}\selectfont$0$};
					\node[][align=left] at (4.25,1.5) {\fontsize{18}{22.4}\selectfont$6$};
					\node[][align=left] at (4.25,2.1) {\fontsize{18}{22.4}\selectfont$6$};
					\node[][align=left] at (4.25,2.7) {\fontsize{18}{22.4}\selectfont$4$};
					\node[][align=left] at (4.25,3.3) {\fontsize{18}{22.4}\selectfont$3$};
					\node[][align=left] at (4.25,3.9) {\fontsize{18}{22.4}\selectfont$3$};
					\node[][align=left] at (4.25,4.5) {\fontsize{18}{22.4}\selectfont$\square$};
					\node[][align=left] at (4.25,5.1) {\fontsize{18}{22.4}\selectfont$2$};
					\node[][align=left] at (4.25,5.7) {\fontsize{18}{22.4}\selectfont$\square$};
					\node[][align=left] at (4.25,6.3) {\fontsize{18}{22.4}\selectfont$0$};
					\node[][align=left] at (4.25,6.9) {\fontsize{18}{22.4}\selectfont$0$};
					\node[][align=left] at (4.25,7.5) {\fontsize{18}{22.4}\selectfont$\square$};
	\end{scope}	
	
	\begin{scope}[xshift = 27cm]
					\draw (3.5,1) -- (5,1);
					\draw (3.5,1) -- (3.5,8);
					\draw (5,1) -- (5,8);
					\draw (5.5, 8.5) -- (10,8.5);
					\node[][align=left] at (4.25,1.5) {\fontsize{18}{22.4}\selectfont$6$};
					\node[][align=left] at (4.25,2.1) {\fontsize{18}{22.4}\selectfont$6$};
					\node[][align=left] at (4.25,2.7) {\fontsize{18}{22.4}\selectfont$4$};
					\node[][align=left] at (4.25,3.3) {\fontsize{18}{22.4}\selectfont$3$};
					\node[][align=left] at (4.25,3.9) {\fontsize{18}{22.4}\selectfont$3$};
					\node[][align=left] at (4.25,4.5) {\fontsize{18}{22.4}\selectfont$\square$};
					\node[][align=left] at (4.25,5.1) {\fontsize{18}{22.4}\selectfont$2$};
					\node[][align=left] at (4.25,5.7) {\fontsize{18}{22.4}\selectfont$\square$};
					\node[][align=left] at (4.25,6.3) {\fontsize{18}{22.4}\selectfont$0$};
					\node[][align=left] at (4.25,6.9) {\fontsize{18}{22.4}\selectfont$0$};
					\node[][align=left] at (4.25,7.5) {\fontsize{18}{22.4}\selectfont$0$};
					\node[][align=left] at (4.25,8.1) {\fontsize{18}{22.4}\selectfont$\square$};
					
	\end{scope}		
\end{tikzpicture}}
\caption{A $(3,1)$-stack shown at several stages.}
\label{F:(3,1)-stack}
\end{figure}
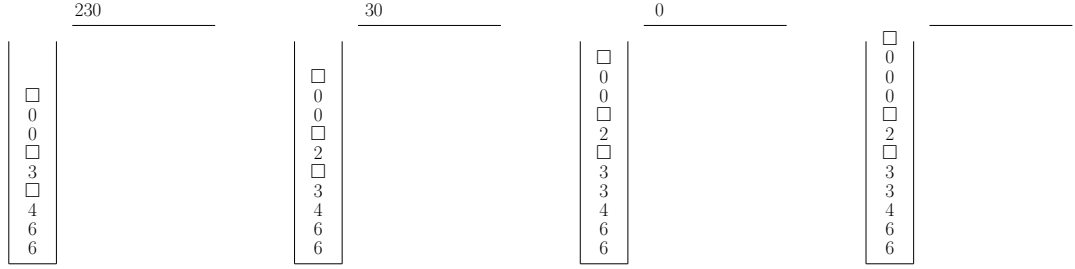

Varying the restrictions of depth and generalizations of the push operation as described above, we get a family of stacks and a family of pop stacks.  There are some variations that no longer give us new machines though.  Specifically, if our stack or pop stack has limited depth, having the option to push entries to (or pop entries from) spaces that do not exist does add any sorting power.

\begin{proposition}  For a $(r,1)$-stack or $(r,1)$-pop stack of depth $d$, if $r>d$, the machine sorts the same words as if $r=d$.
\end{proposition}

As such, we will restrict $r \leq d$.

We will first consider the case where $r=d=2$.

The $(2,1)$-pop stack of depth two has the nice characteristic where it sorts words (and permutations) that avoid three patterns of length three.

\begin{theorem}~\label{T:2-1-2_pop}
A $(2,1)$-pop stack of depth two sorts exactly the words that avoid $120, 201, 210$.
\end{theorem}

\begin{proof}
Because the $(2,1)$-pop-stack we are using only has depth two, any pattern of length three with three distinct values will have at least the first entry popped out before the third entry can enter.  For all of $120, 201, 210$, this means the larger first entry is popped out before the smaller last entry enters the stack.  Thus no word containing $120, 201, 210$ is sortable by a $(2,1)$-pop stack of depth two.

Now consider a word $w$ that cannot be sorted by a $(2,1)$-pop stack of depth two.  If a larger entry was forced out before a smaller entry could be pushed into the stack due to the limited depth, there are two ways this could have happened.   One is that two distinct larger entries appeared before a smaller entry, that is $w$ contains a $120$ or $210$ pattern.  Otherwise, we have a large entry that is below a small entry in the pop stack, and our third entry waiting to enter is between the small and large entry in value.  Because of the $(2,1)$ generalization, in theory, the first two entries could appear in either order in $w$ to form either a $021$ or a $201$ pattern.  However, in the case of a $021$ pattern, the $0$ entry could be popped out before pushing the $2$ into the pop stack, so only the $201$ pattern presents a problem.
\end{proof}

Note:  Algorithm~\ref{A:(2,1)-popD2} gives a procedure for sorting with a $(2,1)$-pop stack of depth $2$.  Because there are only two possible values in $PS2$, it would be equivalent to simply push $w_i$ to the bottom of $PS2$ if it matched the larger value in the $(2,1)$-pop stack of depth $2$ rather than keep track of how many smaller entries there are.  However, we present the algorithm in this form for the option of generalization if desired.

\begin{algorithm}
\caption{$(2,1)$-Pop Stack of Depth $2$ Sorting Algorithm}\label{A:(2,1)-popD2}
\begin{algorithmic}
\State $w=w_1w_2\cdots w_n$ is a word of length $n$.
\State $PS2$ is a $(2,1)$-Pop Stack of Depth $2$ .
\State $O$ is an empty array.
\State $i = 1$  is the index of the next entry in the input.
\State $s = 0$ is the counter for the number of copies of the smallest value in the $(2,1)$-Pop Stack of Depth $2$.
\Procedure{PopStackSort2}{$w,n$}
\While{$i < n+1$}
	\If{$PS2$ is empty}
		\State push $w_i$ to the top of $PS2$
		\State $i = i+1$  
		\State $s = 1$
	\ElsIf{$w_i = $ the top/smallest element(s) in $PS2$}
   		 \State push $w_i$ to the top of $PS2$
		 \State $i = i+1$
		 \State $s = s+1$
	\ElsIf{$w_i = $ the larger element(s) in $PS2$}
   		 \State push $w_i$ to the $(s+1)$st position of $PS2$
		 \State $i = i+1$
	\ElsIf{there are already two distinct values in $PS2$}
		\While{$PS2$ is nonempty}
			\State pop the top entry of $PS2$ to the end of $O$
		\EndWhile
	\ElsIf{$p_i < $ top element of $PS2$}
   		 \State push $p_i$ to the top of $PS2$
		 \State $i = i+1$
		 \State $s = 1$
   	\Else
		\While{$PS2$ is nonempty}
			\State pop the top entry of $PS2$ to the end of $O$
		\EndWhile
	\EndIf
\EndWhile \\
\Return $O$
\EndProcedure
\end{algorithmic}
\end{algorithm}

\begin{theorem} Algorithm~\ref{A:(2,1)-popD2} is optimal for sorting words with an $(r,1)$-pop stack of depth $2$.
\end{theorem}

\begin{proof}
By way of contradiction, let $w$ be a sortable word that is not sorted by our algorithm on a $(2,1)$-pop stack of depth 2.  Specifically, $S(w)$ must contain a descent, say $w_jw_k$.  Because Algorithm~\ref{A:(2,1)-popD2} never allows a descent in the pop stack, $w_j$ must have been the last entry popped from one pop operation and $w_k$ must have been the first entry popped from the next pop operation.  The entries $w_j$ and $w_k$ must have appeared in the same relative order (forming an inversion) in $w$.  The only way to remove this inversion would be for both entries to be in the pop stack at the same time.

Notice Algorithm~\ref{A:(2,1)-popD2} only pops the pop stack in two cases.  In the first case, the pop stack already contained two distinct values that are different from the next input value.  In this case, $w_j$ is one of the larger entries in the pop stack.  Also,  $w_j$ is the first entry in $w$ from the all of entries of the same value popped at the same time as $w_j$ because  a new entry only enters in the pop stack if it is smaller than the other value represented or the new entry matches a value already there and is placed above it in the pop stack.  In fact, $w_j$ must be pushed into an empty pop stack as we are not inserting any unmatched larger entries into the pop stack after a smaller entry.  Thus there must be a smaller entry $w_m$ popped at the same time as $w_j$ and was inserted into the pop stack after $w_j$.  Notice $w_j w_m w_k$ form a $201$ or $210$ pattern in $w$ which is forbidden.

Otherwise, when $w_j$ is popped from the pop stack, only values matching $w_j$ can be in the pop stack.  For the next entry $w_i$, not to enter the pop stack, we must have $w_j < w_i$.  In this case, $w_j w_i w_k$ form a $120$ pattern in $w$ which is also forbidden.
\end{proof}

The enumeration problem for this $(2,1)$-pop stack of depth $2$ can be considered for permutations, words, and specifically inversion sequences. The permutations of length $n$ avoiding $231, 312, 321$, i.e. $120,201,210$, are known as the \emph{layered permutations} where each layer can have size one or two.  Note that layered permutations are exactly the permutations that avoid $231,312$.  Furthermore, the restriction of avoiding $321$ forces each layer (consecutive decreasing subsequence) to have size at most two.   These restricted layered permutations are enumerated by the Fibonacci numbers as shown by Simion and Schmidt~\citeyear{simion:restricted-perm:}.  As a consequence of their result and Theorem~\ref{T:2-1-2_pop} we have the following corollary.

\begin{corollary}  (Simion and Schmidt) The number of permutations of length $n$ sortable by a $(2,1)$-pop stack of depth $2$ is the $(n+1)^{\text{st}}$ Fibonacci number $F_{n+1}$.
\end{corollary}

Burstein~\citeyear[Theorem 5.1]{BB} enumerated the words of length $n$ over an alphabet $[k]$ avoiding $120,201,210$.

\begin{theorem} (Burstein)  The generating function for the number of words of length $n$ over an alphabet $[k]$ sortable by a $(2,1)$-pop stack of depth $2$ is given by the coefficient of $x^n y^k$ in $F(x,y)$ where
$$F(x,y) = \frac{(1-x)(1-2x)-((1-x)(1-2x)+x^2)y}{(1-x)(1-2x)-(1-x)(2-3x)y+(1-2x)y^2}.$$
\end{theorem}

The case for the enumeration of sortable inversion sequences by a $(2,1)$-pop stack of depth $2$ (in the form of avoidance of $120,201,210$ or otherwise) does not appear in the literature prior to this work, but can be solved using generating trees as demonstrated in Section~\ref{S:InvSeq120_201_210}.  Specifically, the generating function for these sortable inversion sequences is given in Theorem~\ref{T:genFun}.

\section{The generating tree method}~\label{S:GenTree}

\subsection{A formula for the generating function for the number of inversion sequences of length $n$ that avoid $120,201,210$}~\label{S:InvSeq120_201_210}
To enumerate the inversion sequences of length $n$ avoiding $120,201,210$, i.e. the inversion sequences sortable by a $(2,1)$-pop stack of depth $2$, we use the generating tree method shown by Kotsireas,  Y\i ld\i r\i m, and the first author~\citeyear{MY}.  See the recent paper of Pantone~\citeyear{Pantone} for another nice description of using succession rules to enumerate pattern avoiding inversion sequences.

Let $\A=\cup_{n=0}^{\infty} \I_n(120,201,210)$. 
We construct a pattern-avoidance tree $\T$ for the class of pattern-avoiding inversion sequences $\A$ as follows.
\begin{enumerate}
\item The root is $0$ (inversion sequence with one letter), that is, $0\in \T$ at level $1$.
\item Recursively construct the nodes at level $n+1$ of tree $\T$ from the nodes at level $n$ by inserting a new letter at the end of the inversion sequence.  That is, the children of $e=e_0\cdots e_n\in\I_n\cap\A$ are the inversion sequences $e^{*}=e_0\cdots e_nj$ with $j=0,1,\ldots,n+1$ where $e^{*}\in\I_{n+1}\cap\A$.
\end{enumerate}

Now, we relabel the vertices of the tree $\T$ as follows. Define $\T(e)$ to be the subtree consisting of the inversion sequence $e$ as the root and its descendants in $\T$.  We say that $e$ is {\em equivalent} to $e'$, denoted by $e \sim e'$, if and only if $\T(e) \cong \T(e')$ (in the sense of plane trees).  Define an order on the nodes of any tree to be from top to bottom, and within a level from left to right.  Denote by $\T'$ the tree $\T$ where we have replaced, in order, each node $e \in T$ by the first node $e'\in\T$ (from top to bottom and within a level from left to right) in $\T$ such that $\T(e)\cong\T(e')$ assuming that such a node $e'$ exists.

\begin{definition}  A \emph{succession rule}, denoted $x \rightsquigarrow y_1,y_2,\ldots,y_k$, will indicate that the set $\{y_1,y_2,\ldots,y_k\}$ will be such that the generating subtrees rooted at $y_1,y_2,\ldots,y_k$ are equivalent to those rooted at \\ $z_1,z_2,\ldots,z_k$ where $\{z_1,z_2,\ldots,z_k\}$ is the complete set of children of $x$.
\end{definition}

Note that exponents are used in two ways below based on context.  Specifically, when describing succession rules, an exponent of $j$ on a child indicates the number of copies of that child the parent has.  For example, in Lemma~\ref{lem1}, $b_{m,j}$ has $j$ children of the form $b_{m+2-j,1}$.  However, when describing the individual nodes that make up the tree, the exponent refers to copies of the letter in the word (in particular, inversion sequence).  Again, referring to Lemma~\ref{lem1}, we have $a_m$ is the word made up of exactly $m$  $0$s.

\begin{lemma}\label{lem1}
The generating tree $\T'$ is given by root $0$ and the following succession rules
\begin{align*}
a_m&\rightsquigarrow a_{m+1},b_{m,1},\ldots,b_{m,m},\\
b_{m,j}&\rightsquigarrow (b_{m+2-j,1})^j,b_{m+1,j},b_{m+1-j,1},\ldots,b_{m+1-j,m+1-j},
\end{align*}
where $a_m=0^m$ and $b_{m,j}=a_mj$ for $1\leq j \leq m$.  
\end{lemma}

\begin{proof}
We label the inversion sequence $0\in\I_0$ by $a_1$. Thus, $a_1\rightsquigarrow a_2,b_{1,1}$. More generally, by the definitions, the children of $a_m\in\T'$ are $a_m0,a_m1,\ldots,a_mm$, which can also be denoted $a_{m+1},b_{m,1},\ldots,b_{m,m}$, respectively. Thus, the rule $a_m\rightsquigarrow a_{m+1},b_{m,1},\ldots,b_{m,m}$ holds.

Also, the children of $b_{m,j}\in\T'$ are $b_{m,j}0,b_{m,j}1,\ldots,b_{m,j}(m+1)$. Notice  $\T(b_{m,j}k)\cong\T(b_{m+2-j,1})$ with $k=0,1,\ldots,j-1$. To show this, we map any inversion sequence $\pi=0^mjk \pi' \in \I_n(120,201,210)$ to $0^{m+2-j}1 \pi''$, where $\pi''$ is obtained from $\pi'$ (there are no letters in $\pi'$ belonging to the set $\{0,\ldots,k-1,k+1,\ldots,j-1\}$) by replacing each letter $x\geq j$ by $x+1-j$ and replacing the letter $k$ by $0$. 
Hence, we see that $\pi\in\I_n(120,201,210)$ if and only if $0^{m+2-j}1\pi''\in I_{n+2-j}(120,201,210)$.

As none of our patterns have repeated letters, it is not hard to see $\T(b_{m,j}j)\cong\T(b_{m+1,j})$.  Here we map any inversion sequence $0^mjj\pi'\in\I_n$ to $0^{m+1}j\pi'\in I_n$ with $1\leq j\leq m$, so this map respects the rule of avoiding $120,201,210$.

In the last cases when $k=j+1,j+2,\ldots,m+1$, we have a similar mapping as with the smaller $k$ values.  Specifically, $\T(b_{m,j}k)\cong\T(b_{m+1-j,k-j})$.
 To see this, again, let $\pi=0^mjk\pi'\in I_n(120,201,210)$. Since $\pi$ avoids $120$, we see that each letter of $\pi'$ is at least $j$. 
Let $\pi''$ be the obtained sequence from $\pi'$ by decreasing each nonzero letter of $\pi'$ by $j$. 
Then the map from $\pi=0^mjk\pi'\in\I_n(120,201,210)$ to $0^{m+1-j}(k-j)\pi''\in\I_{n-j}(120,201,210)$ is a bijection.
 Hence, we have the following rule $b_{m,j}\rightsquigarrow (b_{m+2-j,1})^j,b_{m+1,j},b_{m+1-j,1},\ldots,b_{m+1-j,m+1-j}$, which completes the proof.
\end{proof}

Define $A_m(x)$ (respectively, $B_{m,j}(x)$) to be the generating function for the number of nodes at level $n\geq1$ for the subtree of $\T(B;a_{m})$ (respectively, $\T(B;b_{m,j})$), where the root stays at level $1$. Thus, by Lemma \ref{lem1}, we have
\begin{align}
A_m(x)&=x+xA_{m+1}(x)+x\sum_{j=1}^mB_{m,j}(x),\label{eqA1}\\
B_{m,j}(x)&=x+jxB_{m+2-j,1}(x)+xB_{m+1,j}(x)+x\sum_{k=1}^{m+1-j}B_{m+1-j,k}(x),\label{eqA2}
\end{align}
for all $1\leq j\leq m$.

In order to solve the above recurrence relations, we define
\begin{align*}
A(v)&=\sum_{m\geq1}A_m(x)v^{m-1}, \\
B(v,u)&=\sum_{m\geq1}\sum_{j=1}^mB_{m,j}(x)u^{m-j}v^{m-1}, \text{ and}\\
 C(v)&=\sum_{m\geq1}B_{m,1}(x)v^{m-1}.
 \end{align*}

 Thus Equations \eqref{eqA1}-\eqref{eqA2} can be written as
\begin{align}
A(v)&=\frac{x}{1-v}+\frac{x}{v}(A(v)-A(0))+xB(v,1),\label{eqA3}\\
B(v,u)&=\frac{x}{(1-v)(1-uv)}+\frac{x}{uv(1-v)^2}(C(vu)-C(0))+\frac{x}{uv}(B(v,u)-B(v,0))+\frac{x}{1-v}B(uv,1),\label{eqA4}\\
C(v)&=\frac{x}{1-v}+\frac{2x}{v}(C(v)-C(0))+xB(v,1),\label{eqA5}
\end{align}
where Equation \eqref{eqA5} is the translation of \eqref{eqA2} with $j=1$.

To solve the system from Equations \eqref{eqA4}-\eqref{eqA5}, we make the following guess based on the first terms of the generating functions $B(v,1)$, $C(v)$, and $A(v)$:
\begin{align}
B(v,1)&=\frac{1}{(1-v)^2}C(v)-\frac{v}{(1-v)^2}A(v).\label{eqA6}
\end{align}

Next, we solve the system from Equations \eqref{eqA4}-\eqref{eqA6}, which satisfies the original system from Equations \eqref{eqA4}-\eqref{eqA5}. By substituting the expression of $B(v,1)$ from Equation \eqref{eqA6} into Equation \eqref{eqA4}, and solving for $C(v)$, we obtain
\begin{align}
C(v)&=\frac{v^3-2v^2+2vx+v-x}{vx}A(v)+\frac{(1-v)^2}{v}A(0)+v-1.\label{eqA7}
\end{align}
From here, use Equation \eqref{eqA6} and Equation \eqref{eqA7} to rewrite Equation \eqref{eqA5} as
\begin{align}
&\frac{2x^2-3x(x+1)v+(5x+1)v^2-2(x+1)v^3+v^4}{v^2x}A(v)\notag\\
&=\frac{2x-(3x+1)v+2(x+1)v^2-v^3}{v^2}A(0)-\frac{2x}{v}C(0)-\frac{(v-2x)(v-1)}{v}.
\label{eqA8}
\end{align}
Let $K(v)=2x^2-3x(x+1)v+(5x+1)v^2-2(x+1)v^3+v^4$ be the kernel of this equation. Note that for the kernel equation $K(v)=0$ there are four roots, say $v_1,v_2,v_3,v_4$, where
\begin{align*}
v_1&=1+\frac{-1+\sqrt{5}}{2}x+\frac{-5+4\sqrt{5}}{5}x^2+\cdots,\\
v_2&=1+\frac{-1-\sqrt{5}}{2}x+\frac{-5-4\sqrt{5}}{5}x^2+\cdots,\\
v_3&=2x+2x^2+10x^3+\cdots,\\
v_4&=x-x^3-x^4+\cdots.
\end{align*}
By taking Equation \eqref{eqA8} with either $v=v_3$ or $v=v_4$, we obtain a system of equations in $A(0)$ and $C(0)$. We solve this system, obtaining
\begin{align*}
A(0)&=-\frac{(v_1+v_2-2x-1)v_1v_2}{v_1v_2(v_1+v_2)-2(1+x)v_1v_2+2x}, \\
  C(0)  &=-\frac{v_1v_2(v_2-1)(v_1-1)-(v_1^2v_2+v_1v_2^2-2v_1^2
-3v_1v_2-2v_2^2+2v_1+2v_2)x}{2x(v_1v_2(v_1+v_2)-2(1+x)v_1v_2+2x)}\\
&\qquad + \frac{2(v_1v_2-2v_1-2v_2+2)x^2}{2x(v_1v_2(v_1+v_2)-2(1+x)v_1v_2+2x)}..
\end{align*}
Using the expressions of $A(0),C(0)$, we obtain an explicit formula for $A(v)$ from Equation \eqref{eqA8}. We can then obtain an explicit formula for $C(v)$ from Equation \eqref{eqA7}.  Then we can use that formula to get an explicit formula for $B(v,1)$ from Equation \eqref{eqA6}.  Finally, we have the information we need to obtain an explicit formula for $B(v,u)$ from Equation \eqref{eqA4}. We omit the presentations of the expressions $A(v),C(v),B(v,1),B(v,u)$ because they are very lengthy.  However, these expressions satisfy the system of Equations \eqref{eqA3}-\eqref{eqA5}. Hence, we can state the following result.
\begin{theorem}~\label{T:genFun}
The generating function for the number of inversion sequences in $\I_n$ that avoid $120,201,210$ and thus are sortable by a $(2,1)$-pop stack of depth $2$ is given by
\begin{align*}
A(0)&=-\frac{(v_1+v_2-2x-1)v_1v_2}{v_1v_2(v_1+v_2)-2(1+x)v_1v_2+2x}\\
&=x+2x^2+6x^3+23x^4+101x^5+484x^6+2468x^7+13166x^8+72630x^9+411076x^{10}\\
&+2374188x^{11}+13938018x^{12}+82932254x^{13}+499031324x^{14}+3031610924x^{15}\\
&+18568429963x^{16}+114541486785x^{17}+710973143614x^{18}+4437415155234x^{19}\\
&+27831038618735x^{20}+175318861863701x^{21}+1108762012137252x^{22}\\
&+7037137177329268x^{23}+44808588430903068x^{24}+\cdots.
\end{align*}
\end{theorem}

\subsection{The generating tree for the pop stack sortable inversion sequences}~\label{S:popgen}

As in the previous case, one can show that the generating tree (based on the algorithm given in \cite{MY}) of $I_n(120, 201, 1010)$, that is, of the pop stack sortable inversion sequences, is given by a root $a_1$ and the rules
\begin{align*}
a_m&\rightsquigarrow a_{m+1},b_{m,1},\ldots,b_{m,m},\\
b_{m,j}&\rightsquigarrow c_{m+1-j,1},\ldots,c_{m,j},b_{m+1,j},b_{m+1-j,1},\ldots,b_{m+1-j,m+1-j},\\
c_{m,j}&\rightsquigarrow c_{m+2-j,1},\ldots,c_{m+1,j},a_{m+3-j},b_{m+2-j,1},\ldots,b_{m+2-j,m+2-j},
\end{align*}
where $a_m=0^m$, $b_{m,j}=0^mj$ and $c_{m,j}=0^mj(j-1)$.

For any sequence of nodes $f_{m,j}$, we define $F_{m,j}(x)$ as the generating functions for the number of nodes in the subtrees $\mathcal{T}(B;f_{m,j})$, where $f\in \{a,b,c\},F\in \{A,B,C\}$. Define
\begin{align*}
A(v)&=\sum_{m\geq1}A_m(x)v^{m-1}, \\
B(v,u)&=\sum_{m\geq1}\sum_{i=1}^mB_{m,i}v^{m-1}u^{m-i},\\
C(v,u)&=\sum_{m\geq1}\sum_{i=1}^mC_{m,i}v^{m-1}u^{m-i}.
\end{align*}
Then the generating function for the case {120,201,1010} is $\frac{x}{1-x}+xB(x,1)$, where $B(x,v)$ satisfies
\begin{align*}
A(v)&=\frac{x}{1-v}+\frac{x}{v}(A(v)-A(0))+xB(v,1),\\
B(v,u)&=\frac{x}{(1-v)(1-uv)}+\frac{x}{uv}(B(v,u)-B(v,0))+\frac{x}{1-v}C(v,u) +\frac{x}{1-v}B(vu,1),\\
C(v,u)&=\frac{x}{(1-v)(1-uv)}+\frac{x}{uv(1-v)}(C(v,u)-C(v,0))\\
& +\frac{x}{u^2v^2(1-v)}(A(uv)-A(0)-uv\frac{\partial}{\partial v}A(v)\mid_{v=0}) +\frac{x}{uv(1-v)}(B(vu,1)-B(0,0)).\\
\end{align*}
By applying this system $20$ times starting from $A(x,v)=B(x,v,u)=C(x,v,u)=0$, we obtain the first $20$ coefficients of $A(0)$ as
\begin{align*}
A(0)&=x+2x^2+6x^3+23x^4+101x^5+485x^6+2488x^7+13414x^8+75126x^9+433546x^{10}\\
&+2563335x^{11}+15461646x^{12}+94835817x^{13}+589997530x^{14}+3715451178x^{15}\\
&+23645541066x^{16}+151874732111x^{17}+983428159871x^{18}\\
&+6413887925931x^{19}+42100271440339x^{20}+\cdots.
\end{align*}

\subsection{Inversion sequences sortable by a pop stack of depth two}~\label{S:ps_d2}

Recall from Corollary~\ref{C:popstackD} that the inversion sequences sorted by a pop stack of depth two are precisely those which avoid $P=\{120,201,210,1010\}$. We can define the generating tree $T_P$ for inversion sequences that avoid $P$ as having a root $a_1$ and satisfying the following rules:
\begin{align*}
	a_m & \rightsquigarrow a_{m+1},b_{m,1},\ldots,b_{m,m},\\
	b_{m,j} & \rightsquigarrow (c_{m+1-j})^j,b_{m+1,j},b_{m+1-j,1},\ldots,b_{m+1-j,m+1-j},\quad 1\leq j\leq m,\\
	c_m & \rightsquigarrow c_{m+1},a_{m+2},b_{m+1,1},\ldots,b_{m+1,m+1},
\end{align*}
where $a_m=0^m, b_{m,j}=0^mj$, and $c_m=0^m10$. 
We label the inversion sequence $0$ by $a_1$, so the root of the tree $T_P$ is indeed $a_1$. Now, let us show that the succession rules hold in $T_P$. Based on the algorithm introduced in [17] (for more examples and proofs, see~\cite{CJM, CM23}), we have the following:
\begin{itemize}
\item The children of $a_m$ in $T_P$ are $a_mj$ with $j=0,1,\ldots,m$, which are equal to $a_{m+1},b_{m,1},\ldots,b_{m,m}$. Hence, the first succession rule holds. 
\item The children of $b_{m,j}$ in $T_P$ are $0^mji$ with $i=0,1,\ldots,m+1$. But, note that the set of inversion sequences of the form $0^mji\pi’$ with $0\leq i\leq j-1$ that avoid $P$ can be mapped bijectively to the set of inversion sequences of the form $0^{m+1-j}10\pi''$ that avoid $P$ (by mapping each letter $i$ to $0$ and each letter $s\geq j$ to $s+1-j$ and writing the first initial $0^m$ as $0^{m+1-j}$).  Hence, $T_P(0^mji)\cong T_P(c_{m+1+j})$, for all $i=0,1,\ldots,j-1$. 

Now, let us look at the remaining children of $b_{m,j}$ in $T_P$, which are $0^mji$ with $i=j,j+1,j+2,\ldots,m+1$. Clearly, an inversion sequence $0^mjj\pi’$ avoids $P$ if and only if the inversion sequence $0^{m+1}j\pi'$ avoids $P$, so $T_P(0^mjj)\cong T_P(0^{m+1}j)$. Next consider the inversion sequence $\pi=0^mji\pi’$ with $j+1\leq i\leq m+1$.  Notice that $\pi$ avoids $P$ if and only if $0^{m+1-j}(i-j)\pi''$ avoids $P$ (by mapping each letter $s\geq j$ to $s-j$ and writing the initial $0^mj$ as $0^{m+1-j}$). Hence, the children of $b_{m,j}$ in $T_P$ are $(c_{m+1-j})^j,b_{m+1,j},b_{m+1-j,1},\ldots,b_{m+1-j,m+1-j}$, which proves the second succession rule. 
\item The children of $c_m$ in $T_P$ are $0^m100,0^m101,\ldots,0^m10(m+2)$. Clearly, an inversion sequence $0^m100\pi’$ avoids $P$ if and only if $0^{m+1}10\pi’$ avoids $P$. Also, an inversion sequence $0^m10i\pi’$ avoids $P$ if and only if $0^{m+2}\pi''$ avoids $P$ (here $\pi''$ is obtained by mapping each letter $s$ in $i\pi’$ to $s-1$). Thus the children of $c_m$ in $T_P$ are $c_{m+1},a_{m+2},b_{m+1,1},\ldots,b_{m+1,m+1}$, which shows that the third succession rule holds.  
\end{itemize}

\section{Open Problems}~\label{Open_Problems}

There are many directions left to continue the study of sortable inversion sequences.  The most obvious question is the elusive enumeration problem for the inversion sequences of length $n$ which avoid $120$.  Already interesting as a standalone pattern avoidance question, the fact that these inversion sequences are exactly those sortable by the classic stack sorting algorithm by Knuth~\citeyear{knuth:the-art-of-comp:1} increases the intrigue.  The first author and Shattuck~\citeyear[Section 4]{inversion2015} obtained a recurrence relation for these inversion sequences yielding a functional equation in three variables (sequence number A263778 in OEIS~\cite{OEIS}).  We note that the construction method used here was later generalized by Testart~\citeyear{Testart} as part of his work to complete the enumeration of 22 open cases of pattern-avoiding inversion sequences.

\begin{open} Is there a nicer combinatorial formula for the number of inversion sequences of length $n$ are stack sortable, i.e. avoid $120$?
\end{open}

Next, while we have classified the pop stack sortable inversion sequences and words and provided a generating tree for the inversion sequence case  in Section~\ref{S:popgen}, the following combinatorial enumeration is also still open.

\begin{open} How many inversion sequences of length $n$ are pop stack sortable, i.e. avoid $120,201,1010$?
\end{open}

And the cases restricting pop stacks to particular depths in an analogous way to ~\cite{elder:permutations-ge:, EG2018, EG2021, ELR2016} and/or the expansion in allowable pushes analogous to ~\cite{atkinson:generalized-sta:} can be interesting for words and inversion sequences.

\begin{open} Classify and enumerate the inversion sequences of length $n$ that are $(r,1)$-pop stack sortable for pop stacks of depth $k$ for $r \geq 2$ and/or $k \geq 2$.
\end{open}

Another avenue of exploration is to consider other sorting machines or algorithms.  For instance, the authors considered a sorting algorithm on a single stack prioritizing outputting the next correct entry of the sequence until no other moves were available and then either returning the elements in order to be sorted again~\cite{MSS19} or reversing the output for the next pass~\cite{MSS19Rev}.  In particular, the $2$-reverse-pass sortable permutation class has a reasonably small basis, namely $\{2413, 2431,23154\} = \{1302, 1320, 12043\}$ that may prove tractable.  

\begin{open} How many inversion sequences of length $n$ are $2$-reverse-pass sortable?
\end{open}

\acknowledgements
\label{sec:ack}
The authors are grateful for the helpful comments from the referees.

\bibliographystyle{plainnat}
\bibliography{InvSeq.bib}
\label{sec:biblio}

\end{document}